\documentclass{article}
\usepackage{amssymb,amsmath,amsfonts}
\usepackage{amssymb}
\usepackage{amsfonts,amsmath,epsfig}
\usepackage{mitpress}
\usepackage{graphicx}

\newtheorem{theorem}{Theorem}

\newtheorem{definition}[theorem]{Definition}
\newtheorem{example}[theorem]{Example}

\newtheorem{lemma}[theorem]{Lemma}

\newtheorem{proposition}[theorem]{Proposition}
\newtheorem{remark}[theorem]{Remark}

\newenvironment{proof}[1][Proof]{\noindent\textbf{#1.} }{\ \rule{0.5em}{0.5em}}
\input{tcilatex}
\begin{document}

\title{Commutative rings behind divisible residuated lattices}
\author{Cristina Flaut and Dana Piciu}
\date{}
\maketitle

\begin{abstract}
Divisible residuated lattices are algebraic structures corresponding to a
more comprehensive logic than Hajek's basic logic with an important
significance in the study of fuzzy logic. The purpose of this paper is to
investigate commutative rings whose the lattice of ideals can be equipped
with a structure of divisible residuated lattice. We show that these rings
are multiplication rings. A characterization, more examples and their
connections to other classes of rings are established. Furthermore, we
analyze the structure of divisible residuated lattices using finite
commutative rings. From computational considerations, we present an explicit
construction of isomorphism classes of divisible residuated lattices (that
are not BL-algebras) of small size $n$ ($2\leq n\leq 6$) and we give
summarizing statistics.


\end{abstract}

\section{\textbf{Introduction}}

\bigskip Nonclassical logics are strongly related to computer science and
they have been represented as algebras.

Residuated lattices were introduced in \cite{[Di; 38]}, \cite{[WD; 39]} and
their\ study is originated in the context of theory of rings since the
lattice of ideals of a commutative ring is a residuated lattice, see \cite%
{[Bl; 53]}.

Rings are important algebraic tools and many authors (\cite{[Bl; 53]}, \cite%
{[BN; 09]}, \cite{[LN; 18]}, \cite{[MANH]}, \cite{[FP; 22]}, \cite{[FP; 23]}%
, \cite{[TT; 22]}, etc) have studied algebras of logic connected with
certain classes of rings.

Divisible residuated lattices (or residuated lattice ordered monoids) were
introduced in \cite{[S; 65]} as a unifying concept for Heyting algebras and
Abelian lattice ordered groups. These residuated lattices satisfying
divisibility condition are connected with algebras in t-norm based fuzzy
logics, being examples of BL-algebras.

It is known that BL-rings and MTL-rings are commutative unitary rings whose
ideals form a BL-algebra and respectively an MTL-algebra.

Since BL-algebras are MTL-algebras satisfying divisibility property the
question that one can ask is what are commutative rings whose lattice of
ideals satisfies divisibility condition.

In this paper, we show that divisible residuated lattices are connected with
multiplication rings and using these algebras we obtain a new description
for multiplication rings, see Theorem \ref{t8} and Proposition \ref%
{Proposition_4}.

We show that the class of multiplication rings contains other known classes
of commutative rings: MV-rings, BL-rings, Von Neumann regular rings, rings
which are principal ideal domains and some types of finite unitary
commutative rings, see Proposition \ref{p10} and Theorem \ref{Theorem_2}.

We prove that a commutative ring which has at most $5$\ distinct ideals is
principal and its lattice of ideals is a divisible residuated lattice.
Moreover, the ring $A=A_{1}\times ...\times A_{q}$, such that the numbers of
ideals for the rings $A_{i}$ are $n_{A_{i}}$, with $n_{A_{i}}\in
\{2,3,5\},i\in \{1,2,...,q\}$, is a principal ideal ring, therefore a
multiplication ring, see Remark \ref{r29}.

Using computer algorithms, in \cite{[BV; 10]}, isomorphism classes of
divisible residuated lattices of size $n\leq 12$ were counted. In this
paper, using Theorem \ref{Theorem_4} we construct all (up to an isomorphism)
divisible residuated lattices with $n\leq 6$ elements (that are not
BL-algebras) using the ordinal product BL-algebras. Also, we present
summarizing statistics. This method can be used to construct finite
divisible residuated lattices of larger size, the inconvenience being the
large number of \ algebras that should be generated (for example, there are $%
23$ algebras with $n=6$ \ elements and for $n=7$ there are $49$ divisibile
residuated lattices).

\section{\textbf{Preliminaries}}

\begin{definition}
\label{Definition_1} (\cite{[Di; 38]}, \cite{[WD; 39]})\textbf{\ \ }A \emph{%
(commutative) residuated lattice} \ is an algebra $(L,\wedge ,\vee ,\odot
,\rightarrow ,0,1)$ equipped with an order $\leq $ such that:

(LR1) $\ (L,\wedge ,\vee ,0,1)$ is a bounded lattice;

(LR2) $\ \ (L,\odot ,1)$ is a commutative ordered monoid;

(LR3) $\ z\leq x\rightarrow y$\ iff $x\odot z\leq y,$\ for all $x,y,z\in L.$
\end{definition}

In a residuated lattice for $x\in L,$ we denote $x^{\ast }=x\rightarrow 0.$

\begin{example}
\label{Example_2} It is known that, for a commutative unitary ring $A,$ the
lattice of ideals $(Id(A),\cap ,+,\otimes \rightarrow ,0=\{0\},1=A)$ is a
residuated lattice in which the order relation is $\subseteq $ and 
\begin{equation*}
I+J=<I\cup J>=\{i+j,i\in I,j\in J\}\text{, }
\end{equation*}%
\begin{equation*}
I\otimes J=\{\underset{k=1}{\overset{n}{\sum }}i_{k}j_{k},\text{ }i_{k}\in
I,j_{k}\in J\}\text{, }
\end{equation*}%
\begin{equation*}
I\rightarrow J=(J:I)=\{x\in A,x\cdot I\subseteq J\}\text{,}
\end{equation*}%
\begin{equation*}
Ann\left( I\right) =\left( \mathbf{0}:I\right) \text{, where }\mathbf{0}=<0>,%
\text{ }
\end{equation*}%
\textbf{\ }for every $I,J\in Id\left( A\right) ,$ see \cite{[TT; 22]}. $I+J,$
$I\otimes J,I\rightarrow J$ and $Ann\left( I\right) $ are ideals of $A,$
called sum, product, quotient and annihilator, see \cite{[BP; 02]}.
\end{example}

In a residuated lattice $(L,\wedge ,\vee ,\odot ,\rightarrow ,0,1)$ we
consider the following identities:

\begin{equation*}
(prel)\qquad (x\rightarrow y)\vee (y\rightarrow x)=1\qquad \text{ (\textit{%
prelinearity)}};
\end{equation*}%
\begin{equation*}
(div)\qquad x\odot (x\rightarrow y)=x\wedge y\qquad \text{ (\textit{%
divisibility)}}.
\end{equation*}

\begin{definition}
\label{Definition_2} (\cite{[COM; 00]}, \cite{[CHA; 58]}, \cite{[NL; 03]}, 
\cite{[I; 09]}, \cite{[T; 99]}) \ A residuated lattice $L$ is called:

\begin{enumerate}
\item[$(i)$] \emph{an MTL-algebra}\textit{\ }if $L$ verifies $(prel)$
condition;

\item[$(ii)$] \emph{\ divisible }if $L$ verifies $(div)$ condition;

\item[$(iii)$] \emph{a BL-algebra}\textit{\ }if $L$ verifies $(prel)+(div)$
conditions. A \emph{BL-chain }is a totally ordered BL-algebra, i.e., a
BL-algebra such that its lattice order is total.\medskip \medskip

\item[$(iv)$] \emph{an MV-algebra}\textit{\ }if $L$ is a BL-algebra in which 
$x^{\ast \ast }=x,$ for every $x\in L$.
\end{enumerate}
\end{definition}

\begin{proposition}
\label{Proposition_3} (\cite{[I; 09]}) \textit{\ Let }$(L,\vee ,\wedge
,\odot ,\rightarrow ,0,1)$\textit{\ be a residuated lattice. Then we have
the equivalences:}

\textit{(i) }$L\ $satisfies $(\func{div})$\textit{\ condition;}

\textit{(ii) }$\ \ $For all $x,y\in L$ if $y\leq x$ then there exists $z\in
L $ such that $y=z\odot x.$
\end{proposition}

\begin{example}
\label{Example_3}\emph{(}\cite{[I; 09]}\emph{) }\textit{\ }We give an
example of a divisible residuated lattice which is not a BL-algebra. Let $%
L=\{0,a,b,c,1\}$ with $0\leq a,b\leq c\leq 1,$ but $a,b$ incomparable and
the following operations: 
\begin{equation*}
\begin{array}{c|ccccc}
\rightarrow & 0 & a & b & c & 1 \\ \hline
0 & 1 & 1 & 1 & 1 & 1 \\ 
a & b & 1 & b & 1 & 1 \\ 
b & a & a & 1 & 1 & 1 \\ 
c & 0 & a & b & 1 & 1 \\ 
1 & 0 & a & b & c & 1%
\end{array}%
,\hspace{5mm}%
\begin{array}{c|ccccc}
\odot & 0 & a & b & c & 1 \\ \hline
0 & 0 & 0 & 0 & 0 & 0 \\ 
a & 0 & a & 0 & a & a \\ 
b & 0 & 0 & b & b & b \\ 
c & 0 & a & b & c & c \\ 
1 & 0 & a & b & c & 1%
\end{array}%
.
\end{equation*}
\end{example}

We remark that $(a\rightarrow b)\vee (b\rightarrow a)=c\neq 1.$

\section{Divisible residuated lattices and Multiplication rings}

\begin{lemma}
\label{Lemma_0} \textit{\ Let }$(L,\vee ,\wedge ,\odot ,\rightarrow ,0,1)$%
\textit{\ be a residuated lattice. The following assertions are equivalent:}

\textit{(i) }$L$ is divisible\textit{;}

\textit{(ii) }$\ \ $F\textit{or every \thinspace }$x,y,z\in L,$ $%
z\rightarrow (x\odot (x\rightarrow y))=z\rightarrow (x\wedge y).$\textit{\ }
\end{lemma}

\begin{proof}
$(i)\Rightarrow (ii).$ Obviously.

$(ii)\Rightarrow (i).$ $\ $For $z=$ $x\wedge y$ we obtain $(x\wedge
y)\rightarrow (x\odot (x\rightarrow y))=1,$ so $L$ is divisible, since in a
residuated lattice $x\odot (x\rightarrow y)\leq x\wedge y.$
\end{proof}

\bigskip \medskip

Unitary commutative rings for which the lattice of ideals is an MV-algebra,
a BL-algebra or an MTL-algebra are called MV-rings, BL-rings and
respectively MTL-rings and are introduced in \cite{[BN; 09]}, \cite{[LN; 18]}
and \cite{[MANH]}.

\begin{definition}
\label{Definition_4} (\cite{[A; 76]}) \ Let $A$ be a commutative unitary
ring. An ideal $I\in Id(A)$ is called a multiplication ideal if for every
ideal $J\subseteq I$ there exists an ideal $K$ such that $J=I\otimes K.$ The
ring $A$ is called a multiplication ring if all its ideals are
multiplication ideals.
\end{definition}

Multiplication rings have been studied in \cite{[G; 74]}, \cite{[GM; 75]}
and \cite{[M; 64]}. It is well known that a multiplication ring is a subring
of a cartesian product of Dedekind domains and special primary rings, see 
\cite{[M; 64]}.

Using Proposition \ref{Proposition_3} we deduce that:

\begin{theorem}
\label{t8}A unitary commutative ring is a multiplication ring if an only if
its lattice of ideals is a divisible residuated lattice.
\end{theorem}

Using Lemma \ref{Lemma_0} we obtain a new characterization for
multiplication rings:

\begin{proposition}
\label{Proposition_4} \textit{Let }$A$\textit{\ be a commutative unitary
ring. The following assertions are equivalent:}

(i) $A$ is a multiplication ring;

(ii) $I\otimes (J:I)=I\cap J,$ for every $I,J\in Id(A);$

(iii) $((I\otimes (J:I)):K)=((I\cap J):K),$ for every $I,J,K\in Id(A).$
\end{proposition}

We recall that a residuated lattice $(L,\vee ,\wedge ,\odot ,\rightarrow
,0,1)$ in which $x^{2}=x,$ for all $x\in L,$ is called a \emph{Heyting
algebra}, see \cite{[I; 09]} and \cite{[T; 99]}.

In \cite{[BNM; 10]} was proved that unitary commutative rings for which the
semiring of ideals, under ideal sum and ideal product, are Heyting algebras
are exactly Von Neumann regular rings, i.e. commutative rings $A$ in which
for every element $x\in A$ there exists an element $a\in A$ such that $%
x=a\cdot x^{2}.$

Since MV-algebras, BL-algebras and Heyting algebras are divisible residuated
lattices we have the following result:

\begin{proposition}
\label{p10}MV-rings, BL-rings and Von Neumann regular rings are
multiplication rings.
\end{proposition}

\begin{proposition}
\label{prop_mv} A commutative unitary ring is:

(i) a BL-ring if and only if it is an MTL-ring and a multiplication ring;

(ii) an MV-ring if and only if \ it is an multiplication MTL-ring and $%
Ann(Ann(I))=I,$ for every $I\in Id\left( A\right) .$
\end{proposition}

\begin{proof}
(i) Obviously, since a BL-algebra is an MTL-algebra satisfyng the
divisibility property.

(ii) MV-algebras are BL-algebras in which $x^{\ast \ast }=x,$ for every $x$.
\end{proof}

\begin{theorem}
\label{Theorem_2} (i) A commutative ring which is a principal ideal domain
is a \ multiplication ring;

(ii) A ring factor of a principal ideal domain is a multiplication ring.

(iii) A finite commutative unitary ring of the form $A=\mathbb{Z}%
_{k_{1}}\times \mathbb{Z}_{k_{2}}\times ...\times \mathbb{Z}_{k_{r}}$
(direct product of rings, equipped with componentwise operations), \textit{%
where} $k_{i}=p_{i}^{\alpha _{i}}$, $p_{i}$ \textit{is} \textit{a prime
number,\ is a \ multiplication ring.}
\end{theorem}

\begin{proof}
$(i).$ \textit{Let }$A$\textit{\ be a commutative ring which is a principal
ideal domain and }$I=<i>$, $J=<j>,$ be the principal non-zero ideals
generated by $i,j\in A$ $\backslash \{0\}$.

If $d=$\textit{gcd}$\{i,j\}$, then $i=i_{1}d$ and $j=j_{1}d$, with $1=$%
\textit{gcd}$\{i_{1},j_{1}\}$. We have$\ I\cap J=<ij/d>,$ $I\otimes J=<ij>$
and $\left( I:J\right) =$ $<i_{1}>.$

Thus, $J\otimes (I:J)=<j>\otimes <i_{1}>=<ij/d>=I\cap J$.

If $I=\{0\}$, since $A$ is an integral domain, we have $J\otimes (\mathbf{0}%
:J)=J\otimes Ann(J)=\mathbf{0=0}\cap J=$ $\mathbf{0}\otimes (J:\mathbf{0}),$
for every $J\in Id\left( A\right) \backslash \{0\}.$

$(ii).$ \textit{A ring factor of a principal ideal domain \ }is an MV-ring,
see \cite{[FP; 22]}. Using Proposition \ref{p10}, it is a multiplication
ring.

$(iii).$We apply $(ii)$ and Proposition \ref{p10} using the fact that
BL-rings are closed under finite direct products, see \cite{[LN; 18]}.
Moreover, $\left( Id\left( A\right) ,\cap ,+,\otimes \rightarrow
,\{0\},A\right) $ \textit{is a divisible residuated lattice with }$\overset{r%
}{\underset{i=1}{\prod }}\left( \alpha _{i}+1\right) $ \textit{elements},
since it is an \textit{MV-algebra, see }\cite{[FP; 22]}.
\end{proof}

\begin{example}
\label{Remark_2} 1) Following Theorem \ref{Theorem_2}, the ring of integers $%
(\mathbb{Z},\mathbb{+},\cdot )$ is a multiplication ring since $\mathbb{Z}$
is principal ideal domain.

2) Let $K$ be a field and $K\left[ X\right] $ be the polynomial ring. For $%
f\in K\left[ X\right] $, the quotient ring $A=K\left[ X\right] /\left(
f\right) $ is a multiplication ring since the lattice of ideals of this\
ring is an MV-algebra, see \cite{[FP; 22]}.
\end{example}

\begin{proposition}
\label{p14}If $A$ is a multiplication ring, then for every $I,J,K\in
Id\left( A\right) $ we have:

\begin{enumerate}
\item[$(c_{1})$] $Ann(Ann(Ann(I))\rightarrow I)=\{0\};$

\item[$(c_{2})$] $Ann(Ann(I\rightarrow J))=Ann(Ann(I))\rightarrow
Ann(Ann(J));$

\item[$(c_{3})$] $Ann(Ann(I\otimes J))=Ann(Ann(I))\otimes
Ann[Ann(Ann(I))\cap Ann(J)],$

\item[$(c_{4})$] $Ann(Ann(I\cap J))=Ann(Ann(I))\cap Ann(Ann(J));$

\item[$(c_{5})$] $Ann(J)\subseteq I\Rightarrow I\rightarrow Ann(Ann(I\otimes
J))=Ann(Ann(J));$

\item[$(c_{6})$] $I\otimes (J\cap K))=(I\otimes J)\cap (I\otimes K);$

\item[$(c_{7})$] $I\cap (J+K)=(I\cap J)+(I\cap K).$
\end{enumerate}
\end{proposition}

\begin{proof}
$(c_{1}).$ From $Ann(I)\subseteq Ann(Ann(I))\rightarrow I\Rightarrow
Ann(Ann(Ann(I))\rightarrow I)\subseteq Ann(Ann(I))\Rightarrow
Ann(Ann(Ann(I))\rightarrow I)=Ann(Ann(I))\cap Ann(Ann(Ann(I))\rightarrow I)$ 
$\overset{(\func{div})}{=}Ann(Ann(I))\otimes \lbrack Ann(Ann(I))\rightarrow
Ann(Ann(Ann(I))\rightarrow I)]=$ $Ann(Ann(I))\otimes
Ann[(Ann(Ann(I))\rightarrow I))\otimes Ann(Ann(I))]\overset{(\func{div})}{=}%
Ann(Ann(I)\otimes Ann(Ann(Ann(I))\cap I)=Ann(Ann(I))\otimes Ann(I)=\mathbf{0}%
.$

$(c_{2}).$ We have $Ann(Ann(I\rightarrow J))\subseteq Ann(Ann(I))\rightarrow
Ann(Ann(J))=I\rightarrow Ann(Ann(J)).$ Also, $(I\rightarrow
Ann(Ann(J)))\rightarrow Ann(Ann(I\rightarrow J))=$ $Ann(Ann(I\rightarrow
Ann(Ann(J))))\rightarrow Ann(Ann(I\rightarrow J))\supseteq $ $%
Ann(Ann[(I\otimes (I\rightarrow Ann(Ann(J))))\rightarrow J])\overset{(\func{%
div})}{=}Ann(Ann[(I\cap Ann(Ann(J)))\rightarrow J])\supseteq
Ann(Ann[Ann(Ann(J))\rightarrow J])=A$

$\Rightarrow (I\rightarrow Ann(Ann(J)))\rightarrow Ann(Ann(I\rightarrow
J))=A\Rightarrow I\rightarrow Ann(Ann(J))\subseteq Ann(Ann(I\rightarrow J))$ 
$\Rightarrow Ann(Ann(I\rightarrow J))=I\rightarrow
Ann(Ann(J))=Ann(Ann(I))\rightarrow Ann(Ann(J)).$

$(c_{3}).$ From $Ann(Ann(I\otimes J))\subseteq Ann(Ann(I))\Rightarrow $ $%
Ann(Ann(I\otimes J))=Ann(Ann(I\otimes J))\cap Ann(Ann(I))=Ann[J\rightarrow
Ann(I)]\cap Ann(Ann(I))=$ $Ann(Ann(I))\otimes Ann[Ann(Ann(I))\cap Ann(J)].$

$(c_{4}).$ Obviously, $Ann(Ann(I\cap J))\subseteq Ann(Ann(I)),Ann(Ann(J)).$

Let $K\in Id(A)$ such that $K\subseteq Ann(Ann(I)),Ann(Ann(J)).$ Then $%
Ann(Ann(K))\subseteq Ann[Ann(I)+Ann(J)].$ But $K\subseteq Ann(Ann(K))$ and $%
Ann[Ann(I)+Ann(J)]=Ann(Ann(I))\cap Ann(Ann(J))\Rightarrow K\subseteq
Ann(Ann(I))\cap Ann(Ann(J)).$

$(c_{5}).$ $I\rightarrow Ann(Ann(I\otimes J))=Ann(I\otimes J)\rightarrow
Ann(I)=Ann[I\otimes (I\rightarrow Ann(J))]=Ann(I\cap Ann(J))=Ann(Ann(J)).$

$(c_{6}).$ Clearly $I\otimes (J\cap K)\subseteq (I\otimes J)\cap (I\otimes
K).$

Also, $(I\otimes J)\cap (I\otimes K)=(I\otimes J)\otimes \lbrack (I\otimes
J)\rightarrow (I\otimes K)]=I\otimes \lbrack J\otimes (J\rightarrow
(I\rightarrow (I\otimes K)))]=$ $I\otimes \lbrack J\cap (I\rightarrow
(I\otimes K))]=$ $I\otimes \lbrack (I\rightarrow (I\otimes K))\otimes
((I\rightarrow (I\otimes K))\rightarrow J)].$

But $K\subseteq I\rightarrow (I\otimes K)\Rightarrow (I\rightarrow (I\otimes
K))\rightarrow J\leq K\rightarrow J,$ so $(I\otimes J)\cap (I\otimes
K)\subseteq I\otimes (J\cap K)).$

$(c_{7}).$ Clearly, $I\cap (J+K)\supseteq (I\cap J)+(I\cap K).$ Also, we
have $I\cap (J+K)=(J+K)\otimes \lbrack (J+K)\rightarrow I]=[J\otimes
((J+K)\rightarrow I)]+[K\otimes ((J+K)\rightarrow I)]\subseteq \lbrack
J\otimes (J\rightarrow I)]+[K\otimes (K\rightarrow I)]=(I\cap J)+(I\cap K).$
\end{proof}

\section{Examples of divisible residuated lattices using commutative rings}

In this section we present ways to generate finite\ divisible residuated
lattices using finite commutative rings.

In \cite{[I; 09]}, Iorgulescu studies the influence of the condition $(div)$
on the ordinal product of two BL-algebras.

It is known that, if $\mathcal{L}_{1}=(L_{1},\wedge _{1},\vee _{1},\odot
_{1},\rightarrow _{1},0_{1},1_{1})$ and $\mathcal{L}_{2}=(L_{2},\wedge
_{2},\vee _{2},\odot _{2},\rightarrow _{2},0_{2},1_{2})$ are two BL-algebras
such that $1_{1}=0_{2}$ and $(L_{1}\backslash \{1_{1}\})\cap
(L_{2}\backslash \{0_{2}\})=\oslash ,$ then, the ordinal product of $%
\mathcal{L}_{1}$ and $\mathcal{L}_{2}$ is the residuated lattice $\mathcal{L}%
_{1}\boxtimes \mathcal{L}_{2}=(L_{1}\cup L_{2},\wedge ,\vee ,\odot
,\rightarrow ,0,1)$ where

\begin{equation*}
0=0_{1}\text{ and }1=1_{2},
\end{equation*}%
\begin{equation*}
x\leq y\text{ if }(x,y\in L_{1}\text{ and }x\leq _{1}y)\text{ or }(x,y\in
L_{2}\text{ and }x\leq _{2}y)\text{ or }(x\in L_{1}\text{ and }y\in L_{2})%
\text{ ,}
\end{equation*}

\begin{equation*}
x\rightarrow y=\left\{ 
\begin{array}{c}
1,\text{ if }x\leq y, \\ 
x\rightarrow _{i}y,\text{ if }x\nleq y,\text{ }x,y\in L_{i},\text{ }i=1,2,
\\ 
y,\text{ if }x\nleq y,\text{ }x\in L_{2},\text{ }y\in L_{1}\backslash
\{1_{1}\}.%
\end{array}%
\right.
\end{equation*}%
\begin{equation*}
x\odot y=\left\{ 
\begin{array}{c}
x\odot _{1}y,\text{ if }x,y\in L_{1}, \\ 
x\odot _{2}y,\text{ if }x,y\in L_{2}, \\ 
x,\text{ if }x\in L_{1}\backslash \{1_{1}\}\text{ and }y\in L_{2}.%
\end{array}%
\right.
\end{equation*}%
We recall that the ordinal product is associative, but is not commutative.

\begin{proposition}
\label{Proposition_5} ( \cite{[I; 09]}, Corollary 3.5.10) Let $\mathcal{L}%
_{1}$ and $\mathcal{L}_{2}$ be BL-algebras.

\textit{(i) }\ If $\mathcal{L}_{1}$ is a chain, then the ordinal product $%
\mathcal{L}_{1}\boxtimes \mathcal{L}_{2}$ \textit{\ is a BL-algebra (that is
not an MV-algebra);}

\textit{(ii) }$\ \ $If $\mathcal{L}_{1}$ is not a chain, then the ordinal
product $\mathcal{L}_{1}\boxtimes \mathcal{L}_{2}$ \textit{\ is only a
residuated lattice satisfying divisibility condition.}
\end{proposition}

\begin{remark}
\label{Remark_3} An ordinal product of two BL-chains is a BL-chain.
\end{remark}

In \textbf{Table 1} we recall the structure of finite\ BL-algebras $L$ with $%
2\leq n\leq 5$ elements, see \cite{[FP; 23]}.

\begin{equation*}
\text{\textbf{Table 1:}}
\end{equation*}

\textbf{\ }%
\begin{tabular}{lll}
$\left\vert L\right\vert \mathbf{=n}$ & \textbf{Nr of BL-alg } & \textbf{%
Structure } \\ 
$n=2$ & $1$ & $\left\{ Id(\mathbb{Z}_{2})\text{ (BL-chain)}\right. $ \\ 
$n=3$ & $2$ & $\left\{ 
\begin{array}{c}
Id(\mathbb{Z}_{4})\text{ (BL-chain)} \\ 
Id(\mathbb{Z}_{2})\boxdot Id(\mathbb{Z}_{2})\text{ (BL-chain)}%
\end{array}%
\right. $ \\ 
$n=4$ & $5$ & $\left\{ 
\begin{array}{c}
Id(\mathbb{Z}_{8})\text{ (BL-chain)} \\ 
Id(\mathbb{Z}_{2}\times \mathbb{Z}_{2})\text{ (BL)} \\ 
Id(\mathbb{Z}_{2})\boxdot Id(\mathbb{Z}_{4})\text{ (BL-chain)} \\ 
Id(\mathbb{Z}_{4})\boxdot Id(\mathbb{Z}_{2})\text{ \ (BL-chain)} \\ 
Id(\mathbb{Z}_{2})\boxdot (Id(\mathbb{Z}_{2})\boxdot Id(\mathbb{Z}_{2}))%
\text{ (BL-chain)}%
\end{array}%
\right. $ \\ 
$n=5$ & $9$ & $\left\{ 
\begin{array}{c}
Id(\mathbb{Z}_{16})\text{ (BL-chain)} \\ 
Id(\mathbb{Z}_{2})\boxdot Id(\mathbb{Z}_{8})\text{ (BL-chain)} \\ 
Id(\mathbb{Z}_{2})\boxdot Id(\mathbb{Z}_{2}\times \mathbb{Z}_{2})\text{ (BL)}
\\ 
Id(\mathbb{Z}_{2})\boxdot (Id(\mathbb{Z}_{2})\boxdot Id(\mathbb{Z}_{4}))%
\text{ (BL-chain)} \\ 
Id(\mathbb{Z}_{2})\boxdot (Id(\mathbb{Z}_{4})\boxdot Id(\mathbb{Z}_{2}))%
\text{ (BL-chain)} \\ 
Id(\mathbb{Z}_{2})\boxdot (Id(\mathbb{Z}_{2})\boxdot (Id(\mathbb{Z}%
_{2})\boxdot Id(\mathbb{Z}_{2})))\text{ (BL-chain)} \\ 
Id(\mathbb{Z}_{8})\boxdot Id(\mathbb{Z}_{2})\text{ (BL-chain)} \\ 
(Id(\mathbb{Z}_{4})\boxdot Id(\mathbb{Z}_{2}))\boxdot Id(\mathbb{Z}_{2})%
\text{ (BL-chain)} \\ 
Id(\mathbb{Z}_{4})\boxdot Id(\mathbb{Z}_{4})\text{ (BL-chain)}%
\end{array}%
\right. $%
\end{tabular}

Using the construction of ordinal product, Proposition \ref{Proposition_5},
Remark \ref{Remark_3} and Table 1 we can generate divisibile residuated
lattices (which are not BL-algebras) using commutative rings.

\begin{example}
\label{Example_6} To generate the divisible residuated lattice with 5
elements from Example \ref{Example_3}, we consider the commutative rings $(%
\mathbb{Z}_{2}\times \mathbb{Z}_{2},+,\cdot )$ and$\ (\mathbb{Z}_{2},+,\cdot
).$ For $\mathbb{Z}_{2}\times \mathbb{Z}_{2}$ we obtain the lattice $%
Id\left( \mathbb{Z}_{2}\times \mathbb{Z}_{2}\right) =\{\left( \widehat{0},%
\widehat{0}\right) ,\{\left( \widehat{0},\widehat{0}\right) ,\left( \widehat{%
0},\widehat{1}\right) \},\{\left( \widehat{0},\widehat{0}\right) ,\left( 
\widehat{1},\widehat{0}\right) \},\mathbb{Z}_{2}\times \mathbb{Z}%
_{2}\}=\{O,R,B,E\}$, which is an MV-algebra $(Id\left( \mathbb{Z}_{2}\times 
\mathbb{Z}_{2}\right) ,\cap ,+,\otimes =\cap ,\rightarrow ,0=\{\left( 
\widehat{0},\widehat{0}\right) \},1=\mathbb{Z}_{2}\times \mathbb{Z}_{2})$ \
with the following operations: 
\begin{equation*}
\begin{tabular}{l|llll}
$\rightarrow $ & $O$ & $C$ & $B$ & $E$ \\ \hline
$O$ & $E$ & $E$ & $E$ & $E$ \\ 
$C$ & $B$ & $E$ & $B$ & $E$ \\ 
$B$ & $C$ & $C$ & $E$ & $E$ \\ 
$E$ & $O$ & $C$ & $B$ & $E$%
\end{tabular}%
\text{ and }%
\begin{tabular}{l|llll}
$\otimes $ & $O$ & $C$ & $B$ & $E$ \\ \hline
$O$ & $O$ & $O$ & $O$ & $O$ \\ 
$C$ & $O$ & $C$ & $O$ & $C$ \\ 
$B$ & $O$ & $O$ & $B$ & $B$ \\ 
$E$ & $O$ & $C$ & $B$ & $E$%
\end{tabular}%
\text{ }.
\end{equation*}%
If we consider two BL-algebras isomorphic with $(Id\left( \mathbb{Z}%
_{2}\times \mathbb{Z}_{2}\right) ,\cap ,+,\otimes \rightarrow ,0=\{\left( 
\widehat{0},\widehat{0}\right) \},$ $1=\mathbb{Z}_{2}\times \mathbb{Z}_{2})$
and $(Id\left( \mathbb{Z}_{2}\right) ,\cap ,+,\otimes \rightarrow ,0=\{0\},1=%
\mathbb{Z}_{2})$ and denoted by $\mathcal{L}_{1}=(L_{1}=\{0,a,b,c\},\wedge
_{1},\vee _{1},\odot _{1},\rightarrow _{1},0,c)$ and $\mathcal{L}%
_{2}=(L_{2}=\{c,1\},\wedge _{2},\vee _{2},\odot _{2},\rightarrow _{2},c,1),$
using Proposition \ref{Proposition_5} we generate the divisible residuated
lattice $\mathcal{L}_{1}\boxtimes \mathcal{L}_{2}=(L_{1}\cup
L_{2}=\{0,a,b,c,1\},\wedge ,\vee ,\odot ,\rightarrow ,0,1)$ from Example \ref%
{Example_3}.
\end{example}

\begin{remark}
Using the model from Example \ref{Example_6}, for two BL-algebras $\mathcal{L%
}_{1}$ and $\mathcal{L}_{2}$ we can renoted these algebras to obtain two
BL-algebras $\mathcal{L}_{1}^{\prime }$ and $\mathcal{L}_{2}^{\prime }$
isomorphic with $\mathcal{L}_{1}$ and respectively $\mathcal{L}_{2}$ that
satisfy the conditions imposed by the ordinal product.

We denote by $\mathcal{L}_{1}\boxdot \mathcal{L}_{2}$ the ordinal product $%
\mathcal{L}_{1}^{\prime }\boxtimes \mathcal{L}_{2}^{\prime }.$
\end{remark}

\medskip

From Proposition \ref{Proposition_5}, Remark \ref{Remark_3} and Table 1 we
deduce that:

\begin{theorem}
\label{Remark_4} (i) There are no divisible residuated lattices that are
chains and are not BL-algebras;

(ii) There are no divisible residuated lattices with $n\leq 4$ elements that
are not BL-algebras;

(iii) To generate a divisible residuated lattice with $n\geq 5$ elements
(which is not a BL-algebra) as the ordinal product $\mathcal{L}_{1}\boxtimes 
\mathcal{L}_{2}$ of $\ $two BL-algebras $\mathcal{L}_{1}$ and $\mathcal{L}%
_{2}$ we have the following possibilities:%
\begin{eqnarray*}
&&\mathcal{L}_{1}\text{ is a BL-algebra with }i\text{ elements (that is not
a chain) and } \\
&&\mathcal{L}_{2}\text{ is a BL-algebra with }j\text{ elements, }
\end{eqnarray*}%
for $i,j\geq 2,i+j=n+1,$ $i\neq j$

or%
\begin{eqnarray*}
&&\mathcal{L}_{1}\text{ is a BL-algebra (that is not a chain) with }k\text{
elements and } \\
&&\mathcal{L}_{2}\text{ is a BL-algebra with }k\text{ elements, }
\end{eqnarray*}%
for $k\geq 2,$ $k=\frac{n+1}{2}$ $\in N.$
\end{theorem}

We make the following notations: 
\begin{equation*}
\mathcal{BL}_{n}=\text{the set of BL-algebras with }n\text{ elements}
\end{equation*}%
\begin{equation*}
\mathcal{BL}_{n}(c)=\text{the set of BL-chains with }n\text{ elements}
\end{equation*}%
\begin{equation*}
\mathcal{DIV}_{n}=\text{the set of divisible residuated lattices with }n%
\text{ elements}
\end{equation*}%
\begin{equation*}
\mathcal{DIV}_{n}(c)=\text{the set of divisible and totally ordered
residuated lattices with }n\text{ elements.}
\end{equation*}

\begin{theorem}
\label{Theorem_4} (i) All finite divisible residuated lattices (up to an
isomorphism) with $2\leq n\leq 6$ elements can be generated using the
ordinal product of BL-algebras.

(ii) The number of non-isomorphic divisible residuated lattices with $n$
elements ( $2\leq n\leq 6$) is%
\begin{equation*}
\left\vert \mathcal{DIV}_{2}\right\vert =\left\vert \mathcal{BL}%
_{2}\right\vert =1,
\end{equation*}%
\begin{equation*}
\left\vert \mathcal{DIV}_{3}\right\vert =\left\vert \mathcal{BL}%
_{3}\right\vert =2,
\end{equation*}%
\begin{equation*}
\left\vert \mathcal{DIV}_{4}\right\vert =\left\vert \mathcal{BL}%
_{4}\right\vert =5,
\end{equation*}%
\begin{equation*}
\left\vert \mathcal{DIV}_{5}\right\vert =\left\vert \mathcal{BL}%
_{5}\right\vert +1=9+1=10,
\end{equation*}%
\begin{equation*}
\left\vert \mathcal{DIV}_{6}\right\vert =\left\vert \mathcal{BL}%
_{6}\right\vert +3=20+3=23,
\end{equation*}%
(iii) For every $n\geq $ $2,$%
\begin{equation*}
\left\vert \mathcal{DIV}_{n}(c)\right\vert =\left\vert \mathcal{BL}%
_{n}(c)\right\vert .
\end{equation*}
\end{theorem}

\begin{proof}
$(i),(ii).$ \ We recall that $\left\vert \mathcal{BL}_{2}\right\vert
=1,\left\vert \mathcal{BL}_{3}\right\vert =2,\left\vert \mathcal{BL}%
_{4}\right\vert =5,\left\vert \mathcal{BL}_{5}\right\vert =9$ and $%
\left\vert \mathcal{BL}_{6}\right\vert =20,$ see \cite{[BV; 10]} and \cite%
{[FP; 23]}. From Proposition \ref{Proposition_5}, Remark \ref{Remark_3} and
Theorem \ref{Remark_4}, we remark that using the ordinal product of two
BL-algebras we can generate divisible residuated lattices which are not
BL-algebras only for $n\geq 5$.

\textbf{Case }$n=5.$

We have obviously an only divisible residuated lattice that is not a
BL-algebra (up to an isomorphism) since we have a only BL-algebra \ with 4
elements (which is not a chain) isomorphic with $Id\left( \mathbb{Z}%
_{2}\times \mathbb{Z}_{2}\right) $ and a only BL-algebra \ with 2 elements
isomorphic with $Id\left( \mathbb{Z}_{2}\right) ,$ see Example \ref%
{Example_6}.

\textbf{Case }$n=6.$

Using Theorem \ref{Remark_4}, to generate a divisible residuated lattice
with $6$ elements that is not a BL-algebra as an ordinal product $\mathcal{L}%
_{1}\boxtimes \mathcal{L}_{2}$ of $\ $two BL-algebras $\mathcal{L}_{1}$ and $%
\mathcal{L}_{2}$ we consider:%
\begin{eqnarray*}
&&\mathcal{L}_{1}\text{ a BL-algebra with }5\text{ elements (which is not a
chain) and } \\
&&\mathcal{L}_{2}\text{ a BL-algebra with }2\text{ elements}
\end{eqnarray*}%
and%
\begin{eqnarray*}
&&\mathcal{L}_{1}\text{ a BL-algebra with }4\text{ elements (which is not a
chain) and } \\
&&\mathcal{L}_{2}\text{ a BL-algebra with }3\text{ elements.}
\end{eqnarray*}%
We obtain the following 3 algebras: 
\begin{equation*}
(Id(\mathbb{Z}_{2})\boxdot Id(\mathbb{Z}_{2}\times \mathbb{Z}_{2}))\boxdot
Id(\mathbb{Z}_{2}),
\end{equation*}%
\begin{equation*}
Id(\mathbb{Z}_{2}\times \mathbb{Z}_{2})\boxdot (Id(\mathbb{Z}_{2})\boxdot Id(%
\mathbb{Z}_{2}))
\end{equation*}%
\begin{equation*}
Id(\mathbb{Z}_{2}\times \mathbb{Z}_{2})\boxdot Id(\mathbb{Z}_{4}).
\end{equation*}

We conclude that the structure of divisible residuated lattices $L$ (that
are not BL-algebras) with $2\leq n\leq 6$ is (up to an isomorphism):

\medskip 

\begin{tabular}{lll}
$\left\vert L\right\vert \mathbf{=n}$ & \textbf{Nr of algebras } & \textbf{%
Structure} \\ 
$n=2$ & $0$ & $-$ \\ 
$n=3$ & $0$ & $-$ \\ 
$n=4$ & $0$ & $-$ \\ 
$n=5$ & $1$ & $Id(\mathbb{Z}_{2}\times \mathbb{Z}_{2}))\boxdot Id(\mathbb{Z}%
_{2})$ \\ 
$n=6$ & $3$ & $\left\{ 
\begin{array}{c}
\text{ }(Id(\mathbb{Z}_{2})\boxdot Id(\mathbb{Z}_{2}\times \mathbb{Z}%
_{2}))\boxdot Id(\mathbb{Z}_{2}) \\ 
Id(\mathbb{Z}_{2}\times \mathbb{Z}_{2})\boxdot (Id(\mathbb{Z}_{2})\boxdot Id(%
\mathbb{Z}_{2})) \\ 
Id(\mathbb{Z}_{2}\times \mathbb{Z}_{2})\boxdot Id(\mathbb{Z}_{4})%
\end{array}%
\right. $%
\end{tabular}

\medskip 

$(iii).$ Clealy, since a totally ordered divisible residuated lattice is a
BL-algebra.
\end{proof}

\textbf{\ }

\textbf{Table 2 }presents a summary for the number of BL-algebras and
divisibile residuated lattices with $n\leq 6$ elements:

\begin{equation*}
\text{\textbf{Table 2}}
\end{equation*}

\medskip 
\begin{tabular}{llllll}
& $n=2$ & $n=3$ & $n=4$ & $n=5$ & $n=6$ \\ 
BL-algebras & $1$ & $2$ & $5$ & $9$ & $20$ \\ 
divisibile residuated lattices & $1$ & $2$ & $5$ & $10$ & $23$%
\end{tabular}

\textbf{\ \ }

\section{\textbf{Conclusions}}

For a commutative and unitary ring A we have that $(Id(A),\cap ,+,\otimes
\rightarrow ,0=\{0\},1=A)$ is a residuated lattice in which the order
relation is $\subseteq $, $I\rightarrow J=(J:I)$ and $I\odot J=I\otimes J,$
for every $I,J\in Id(A),$ see \cite{[TT; 22]}.

\begin{definition}
Let $A$ be a commutative unitary ring.

(i) The ideal $M$ of the ring $A$ is \textit{maximal} if it is maximal
amongst all proper ideals of the ring $A$. From here, we have there are no
other ideals different from $A$ containing $M$. The ideal $J$ of the ring $A$
is a \textit{minimal ideal} if it is a nonzero ideal which contains no other
nonzero ideals.

(ii) A commutative \textit{local ring} $A$ is a ring with a unique maximal
ideal.

(iii) Let $P$\thinspace $\not=A$ be an ideal in the ring $A$. Let $a,b\in A$
such that $ab\in P$. If we have \thinspace $a\in P$ or $b\in P$, therefore $%
P $ is called a \textit{prime} ideal of $A$.
\end{definition}

\medskip\ 

\begin{definition}
(i) (\cite{[LN; 18]}) A commutative ring $A$ is called a \textit{Noetherian
ring} if the condition of ascending chain is satisfied, that means every
increasing sequence of ideals $I_{1}\subseteq I_{2}\subseteq ...\subseteq
I_{r}\subseteq ...$ is stationary, that means there is $q$ such that $%
I_{q}=I_{q+1}=...$.

(ii) A commutative ring $A$ is called an \textit{Artinian ring} if the
condition of descending chain is satisfied, that means every decreasing
sequence of ideals $I_{1}\supseteq I_{2}\supseteq ...\supseteq
I_{r}\supseteq ...$ is stationary, that means there is $q$ such that $%
I_{q}=I_{q+1}=...$.
\end{definition}

\medskip

\begin{remark}

(i) (\cite{[AB; 19]}, Lemma 3.5) Let $A=\underset{j\in J}{\prod }A_{i}$ be a
direct product of rings. $A$ is a multiplication ring if and only if $A_{j}$
is a multiplication ring for all $j\in J$.

(ii) (\cite{[AB; 19]}, Lemma 3.6) Let $A$ be a multiplication ring and $I$
be an ideal of $A$. Therefore, \thinspace the quotient ring $A/I$ is a
multiplication ring.
\end{remark}

\medskip

\begin{remark}
( \cite{[A; 76]}, Corollary 6.1) Let $A$ be a ring. The following conditions
are equivalent:

(i) $A[X]$ is a multiplication ring;

(ii) $A$ is a finite direct product of fields.
\end{remark}

\medskip

\begin{remark}
(\cite{[AF; 92]} and \cite{[AM; 69]})

1) If $A$ is a Noetherian ring, therefore the polynomial ring $A\left[ X%
\right] $ is Noetherian an the quotient ring $A/I$ is also a Noetherian
ring, for $I~$an ideal of $A$.

2) Any field and any principal ideal ring is a Noetherian ring.

3) Every ideal of the Noetherian ring $A$ is finitely generated.

4) An integral domain $A$ is Artinian ring if and only if $A$ is a field.

5) The ring $K\left[ X\right] /\left( X^{t}\right) $ is Artinian ring, for $%
K $ a field and $t$ a positive integer.

6) A commutative Noetherian ring $A$ is Artinian if and only if $A$ is a
product of local rings.

7) In an Artinian ring every prime ideal is maximal.

8) An Artinian ring is a finite direct product of Artinian local rings.
\end{remark}

\medskip

\begin{proposition}
(\cite{[CFP; 23]}) \textit{Let} $A$ \textit{be a commutative and unitary
ring with a finite number of ideals. Let} $n_{m}\left( A\right) $ \textit{be
the number of maximal ideals in} $A$\textit{,} $n_{p}\left( A\right) $ 
\textit{be the number of prime ideals in} $A$ \textit{and} $n_{I}\left(
A\right) $ \textit{be the number of all ideals in} $A$\textit{. Therefore,} $%
n_{m}\left( A\right) =n_{p}\left( A\right) =\alpha $ \textit{and} $%
n_{I}\left( A\right) =\underset{j=1}{\overset{\alpha }{\prod }}\beta
_{j},\beta _{j}$ \textit{positive integers, }$\beta _{j}\geq 2$.
\end{proposition}

\begin{remark}
1) First, we must remark that $n_{I}\left( A_{j}\right) $ can be any
positive integer $\beta _{j}\geq 2\,$. For example, the ring $K\left[ X%
\right] /\left( X^{\beta _{j}}\right) $ has $\beta _{j}+1$ ideals, $\beta
_{j}\geq 2$.

2) If $n_{I}\left( A\right) $ is finite, from the above proposition, we have
only the following two possibilities:

-$A$ is an integral domain, therefore it is a field and, in this case, we
have $n_{m}\left( A\right) =n_{p}\left( A\right) =1$ and $n_{I}\left(
A\right) =2~$or

- $A$ is not an integral domain and $n_{m}\left( A\right) =n_{p}\left(
A\right) \geq 1$ and $n_{I}\left( A\right) >2$.

3) From the above proposition it is clear that there are not finite
commutative unitary rings such that $\left( n_{m}\left( A\right)
,n_{p}\left( A\right) ,n_{I}\left( A\right) \right) =\left( 3,3,5\right) $
or $\left( n_{m}\left( A\right) ,n_{p}\left( A\right) ,n_{I}\left( A\right)
\right) =\left( 2,2,5\right) $, since $5$ is a prime number. To find such an
example, we must search in infinite rings or in non-commutative rings.
Therefore we have examples only in the case $\left( n_{m}\left( A\right)
,n_{p}\left( A\right) ,n_{I}\left( A\right) \right) =\left( 1,1,5\right) $.
Also, we can find examples of finite commutative unitary rings such that $%
\left( n_{m}\left( A\right) ,n_{p}\left( A\right) ,n_{I}\left( A\right)
\right) =\left( 2,2,4\right) $ or finite commutative unitary rings such that 
$\left( n_{m}\left( A\right) ,n_{p}\left( A\right) ,n_{I}\left( A\right)
\right) =\left( 1,1,4\right) $.
\end{remark}

$\smallskip $

\begin{proposition}
\textit{\ Let} $A$ \textit{be a commutative ring which has at most }$5$%
\textit{\ distinct ideals. Therefore, }\thinspace $A$ \textit{is a principal
ring}.
\end{proposition}

\smallskip

\begin{proof}
Since the ring has a finite number of ideals, it is Noetherian, therefore
all ideals are finite generated. If $A$ has $5$ ideals, let $a,b\in A$ two
distinct non-zero elements and $\left( a\right) $ and $\left( b\right) $ the
ideals generated by $a$ and $\dot{b}$. We suppose that $A$ is not a
principal ideal. Therefore, we have $0,\left( a\right) ,\left( b\right) ,A$
four ideals. It is clear that $\left( a+b\right) $ and $\left( a,b\right) $
are also ideals of $A$. If we have $\left( a\right) =\left( a+b\right) $ it
results that there is $q\in A$ such that $a+b=qa$, then $b\in \left(
a\right) $, false. The same situation is if we suppose that $\left( b\right)
=\left( a+b\right) $. It is clear that $\left( a+b\right) \neq A$, since we
assumed that $A$ is not principal, therefore $\left( a+b\right) =\left(
a,b\right) $ and $A$ is principal.
\end{proof}

\smallskip

\begin{remark}
\label{r29}i) We remark that if $A$ has $2,3$ or $5$ ideals, therefore it is
local, since has only one maximal ideal. If $\ A~$has $4$ ideals, then it is
not local, in general, since we can have two maximal ideals.

ii) Since a commutativ ring which is a principal ideal domain is a
multiplication ring, we have that for a ring with at most $5$ ideals, the
lattice of ideals $(Id(A),\cap ,+,\otimes \rightarrow ,0=\{0\},1=A)$ is a
divisible residuated lattice.

iii) We know that if a ring $A$ has a finite number of ideals, then it is
Artinian and an Artinian ring is a finite direct product of Artinian local
rings. Therefore, if $A=A_{1}\times ...\times A_{q},$ with $A_{i}\,\ $local
Artinian rings, we have that each ideal $J$ in $A$ is of the form $%
J=J_{1}\times ...\times J_{q}$, with $J_{i}$ ideal in $A_{i}$. We obtain
that the ring $A=A_{1}\times ...\times A_{q}$, such that the numbers of
ideals for the rings $A_{i}$ are $n_{A_{i}}$, with $n_{A_{i}}\in
\{2,3,5\},i\in \{1,2,...,q\}$, is a principal ideal ring, therefore a
multiplication ring. For example, if $A=A_{1}\times A_{2}$, with $%
n_{A_{1}}=2 $ and $n_{A_{2}}=5$ is a principal ideal ring with $10$ ideals,
therefore a multiplication ring and the lattice $Id(A)$ is a divisible
residuated lattice. If $A=A_{1}\times A_{2}\times A_{3}$, with $%
n_{A_{1}}=3,n_{A_{2}}=3$ and $n_{A_{3}}=5$, is a principal ideal ring with $%
45$ ideals and $Id(A)$ is a divisible residuated lattice, etc.
\end{remark}

\begin{figure}[tbph]
\centerline{\includegraphics[width=3.7in, height=1.5in]{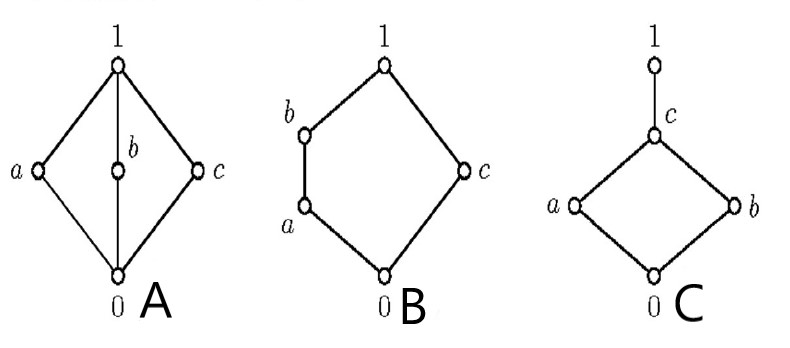}}
\caption{Latices with five elements.}
\end{figure}
\begin{figure}[tbph]
\centerline{\includegraphics[width=3in, height=1.5in]{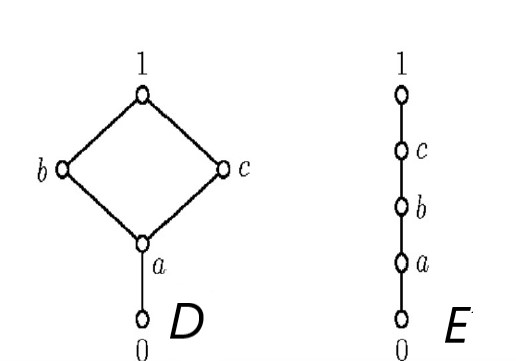}}
\caption{Latices with five elements.}
\end{figure}

\begin{remark}
Since we have only five types of lattices with $5$ elements, from \cite%
{[CFP; 23]}, it results that to obtain divisible residuated lattices (DRL),
the lattice of ideals must be of the type $(Type\_A)\,,(Type\_B)$ or $%
(Type\_C)$. \ $(Type\_B)$ is the pentagon lattice which is not modular and
can be excluded. It remains only $(Type\_$ $A)$, the diamond lattice(which
is not distributive but it is modular), or $(Type\_C)$. But the $(Type\_A)$
involve the existence of a ring with $5$ ideals and three of them to be
maximal ideals, which is false, from the Remark 8. Therefore remains only
the $(Type\_C)$. From here, we have that to find an example of \ DRL which
comes from a multiplication ring and a MTL ring which is not a
multiplication ring, we must search between infinite commutative and unitary
rings and not between finite commutative and unitary rings.
\end{remark}

\smallskip

\begin{definition}
(\cite{[LM; 71]}, p. 150) A unitary ring $A$ is called \textit{arithmetical
ring} if and only if for all ideals $I,J,K$ we have $I\cap \left( J+K\right)
=(I\cap J)+(I\cap K)$, or, equivalently $I+\left( J\cap K\right) =\left(
I+J\right) \cap \left( I+K\right) $, that means the lattice of ideals is
distributive.
\end{definition}

\smallskip

\begin{remark}
1) Using Proposition \ref{p14}, multiplication rings are arithmetical rings

2)In the paper \cite{[MANH]}, the authors studied the class of commutative
rings having the lattice of ideals an MTL-algebra which is not necessary a
BL-algebra. They proved that a local commutative ring with identity is an
MTL-ring if and only if it is an arithmetical ring. Since a MTL-ring is
arithmetical ring and the converse is true if the ring is Noetherian and a
unitary commutative Noetherian MTL-ring is also a BL-ring, therefore, to
obtain a valid example of MTL-ring which is not BL-ring, we must search
outside the Noetherian case.\smallskip

3) From the above, to find an example of MTL-ring $A$ of order $5$ which is
not a BL-ring, the lattice $Id\left( A\right) $ must be on the $(Type\_C)$
since the $(Type\_A)$ is not distributive.\smallskip

In the following, we provide an example of a noncommutative ring $A$ with
the lattice $Id\left( A\right) $ is the diamond lattice.
\end{remark}

\smallskip

\begin{example}
We consider the noncommutative ring $A=\mathcal{M}_{2}\left( \mathbb{Z}%
_{2}\right) $, the set of quadratic matrices over the field $\mathbb{Z}_{2}$%
. We know that $A$ is semisimple, then all right or left ideals are
principal and are generated by idempotents (\cite{[AF; 92]}, p.150). If $%
e\in A$ is an idempotent, then $1-e$ is also an idempotent. We have the
following idempotents: $\{\left( 
\begin{array}{cc}
1 & 0 \\ 
0 & 0%
\end{array}%
\right) ,\left( 
\begin{array}{cc}
0 & 0 \\ 
0 & 1%
\end{array}%
\right) ,\left( 
\begin{array}{cc}
1 & 1 \\ 
0 & 0%
\end{array}%
\right) ,\left( 
\begin{array}{cc}
0 & 1 \\ 
0 & 1%
\end{array}%
\right) ,\left( 
\begin{array}{cc}
1 & 0 \\ 
1 & 0%
\end{array}%
\right) ,\left( 
\begin{array}{cc}
0 & 0 \\ 
1 & 1%
\end{array}%
\right) \}$. We have three proper left ideals: $a=A\left( 
\begin{array}{cc}
1 & 0 \\ 
0 & 0%
\end{array}%
\right) ,b=A\left( 
\begin{array}{cc}
0 & 1 \\ 
0 & 1%
\end{array}%
\right) ,c=A\left( 
\begin{array}{cc}
1 & 0 \\ 
1 & 0%
\end{array}%
\right) $. The lattice $Id\left( A\right) $ is the diamond lattice which is
not a residuated lattice. The ellements $a,b,c$ do not verify the condition $%
3$ from definition of residuated lattices. But the conditions $\left(
prel\right) $ and $\left( div\right) $ are verified.
\end{example}

\bigskip

\begin{equation*}
\end{equation*}

As a further research, we intend to generate MTL residuated lattices which
are not BL-algebras and to provide an example of MTL-algebras of order $5$
which is not BL-algebra.

Cristina Flaut

{\small Faculty of Mathematics and Computer Science, Ovidius University,}

{\small Bd. Mamaia 124, 900527, Constan\c{t}a, Rom\^{a}nia,}

{\small \ http://www.univ-ovidius.ro/math/}

{\small e-mail: cflaut@univ-ovidius.ro; cristina\_flaut@yahoo.com}

\bigskip

Dana Piciu

{\small Faculty of \ Science, University of Craiova, }

{\small A.I. Cuza Street, 13, 200585, Craiova, Romania,}

{\small http://www.math.ucv.ro/dep\_mate/}

{\small e-mail: dana.piciu@edu.ucv.ro, piciudanamarina@yahoo.com}

\end{document}